\RequirePackage{fix-cm}
\documentclass[smallextended]{svjour3}       
\smartqed  
\usepackage{graphicx}
\usepackage{mathptmx}      
\usepackage{amsmath}
\usepackage{amssymb}
\usepackage{nicefrac}
\newcommand{\dualV}[2]{\langle #1, #2 \rangle_{V^*\times V}}

\newcommand{\dualVi}[2]{\langle #1, #2 \rangle_{V_{\ell}^*\times V_{\ell}}}
\newcommand{\dualX}[2]{\langle #1\, ,#2 \rangle_{X^{*}\times X}}
\newcommand{\dualHH}[2]{\langle #1, #2 \rangle_{H^{-1}(\Omega)\times H^{1}_{0}(\Omega)}}

\newcommand{\seq}[1]{\{#1_n\}_{n\in \N}}
\newcommand{\ska}[2]{\left(#1,#2\right)}
\newcommand{\Wp}{W^{1,p}(\Omega)}
\newcommand{\D}{\mathcal{D}'(\Omega)}
\newcommand{\dx}{\,\mathrm{d}x}
\newcommand{\ee}[1]{\mathrm{e}^{#1}}
\newcommand{\sfrac}[2]{\mbox{\footnotesize$\displaystyle\frac{#1}{#2}$}}
\newcommand{\R}{\mathbb{R}}
\newcommand{\N}{\mathbb{N}}
\newcommand{\incl}{\hookrightarrow}
\newcommand{\Lq}{p}
\DeclareMathOperator{\graph}{graph}
\newcommand{\dom}{{D}}
\newcommand{\ran}{{R}}
\renewcommand{\phi}{\varphi}

\newtheorem{ass}{Assumption}

\journalname{}

\begin{document}

\title{Convergence analysis of domain decomposition based time integrators for degenerate parabolic equations
\thanks{This work was funded by CRC 901 Control of self-organizing nonlinear systems: Theoretical methods and concepts of application.}
}
\titlerunning{Domain decomposition integrators}

\author{Monika Eisenmann        \and
             Eskil Hansen%
}

\institute{
  Monika Eisenmann \at
  Institut f\"{u}r Mathematik, Technische Universit\"{a}t Berlin, Stra\ss e des 17.\ Juni 136, 10623 Berlin, Germany\\
  \email{meisenma@math.tu-berlin.de} 
  \and
  Eskil Hansen \at
  Centre for Mathematical Sciences, Lund University, P.O.\ Box 118, 221 00 Lund, Sweden \\
  \email{eskil.hansen@na.lu.se}
}
\date{\today}
\maketitle

\begin{abstract}
Domain decomposition based time integrators allow the usage of parallel and distributed hardware, making them well-suited for the temporal discretization of parabolic systems, in general, and degenerate parabolic problems, in particular. The latter is due to the degenerate equations' finite speed of propagation. In this study, a rigours convergence analysis is given for such integrators without assuming any restrictive regularity on the solutions or the domains. The analysis is conducted by first deriving a new variational framework for the domain decomposition, which is applicable to the two standard degenerate examples. That is, the $p$-Laplace and the porous medium type vector fields. Secondly, the decomposed vector fields are restricted to the underlying pivot space and the time integration of the parabolic problem can then be interpreted as an operators splitting applied to a dissipative evolution equation. The convergence results then follow by employing elements of the approximation theory for nonlinear semigroups.
\end{abstract}

\newpage

\section{Introduction}\label{sec:intro}

Nonlinear parabolic equations of the form 
\begin{equation} \label{eq:pde}
\partial u/\partial t = \nabla\cdot \bigl(D(u,\nabla u)\nabla u\bigr)\quad\text{on }\Omega\times (0,T),
\end{equation}
equipped with suitable boundary and initial conditions, are frequently encountered in applications. If the diffusion constant $D(u,\nabla u)$ vanishes for some values of $u$ and $\nabla u$, i.e., the equation is degenerate, one obtains a quite different dynamics compared to the linear case. The two main nonlinear features are finite speed of propagation and the absence of parabolic smoothening of the solution. Concrete applications can, e.g., be found when modelling gas flow through porous media, phase transitions and population dynamics. A survey of such applications is given in \cite[Section~1.3 and Chapter~2]{Vasquez.2007}. In order to keep the presentation as clear-cut as possible, we will mostly ignore the presence of lower-order advection and reactions terms. 

Approximating the solution of a partial differential equation typically results in large-scale computations, which require the usage of parallel and distributed hardware. One possibility to design numerical schemes that make use of such hardware is to decompose the equation's domain into a family of subdomains. The domain decomposition method then consists of an iterative procedure where, in every step, the equation is solved independently on each subdomain and the resulting solutions are thereafter communicated to the adjacent subdomains. This independence of the decomposed equations and the absence of global communication enables the parallel and distributed implementation of domain decomposition methods. For linear parabolic equations the common procedure is to first discretize the equation in time by a standard implicit integrator. Then an elliptic equation on $\Omega$ is obtained in every time step, which is iteratively solved by a domain decomposition based discretization. We refer to the monographs~\cite{Mathew.2008,QuarteroniValli.1999,ToselliWidlund.2005} for an in-depth treatment of this approach.  Another possibility is to apply the domain decomposition method to the full space-time domain $\Omega\times (0,T)$, which leads to an iterative procedure over parabolic problems that can be parallelized both in space and time; see, e.g., \cite{Gander.1999,GanderHalpern.2007,GiladiKeller.2002}. 

When considering nonlinear parabolic problems one finds that there are hardly any results concerning the analysis of domain decomposition based schemes. Two exceptions are the papers~\cite{KimEtal.2000,Lapin.1991}, where domain decomposition schemes are analyzed for non-degenerate quasilinear parabolic equations and the degenerate two-phase Stefan problem, respectively. The lack of results in the context of degenerate equations is rather surprising from a practical point of view, as the equations' finite speed of propagation is ideal for applying domain decomposition strategies. For example, a solution that is initially zero in parts of the domain $\Omega$ will in each time step only propagate to a small number of neighboring subdomains, which limits the computational work considerably. However, from a theoretical perspective the lack of convergence results is less surprising. The issue is that the standard domain decomposition schemes all link together the equations on the subdomains via boundary conditions. As the solutions of degenerate parabolic equations typically lack higher-order regularity, making sense of such boundary linking is, at the very least, challenging.

\begin{figure}
 \centering 
 \includegraphics[scale=0.6]{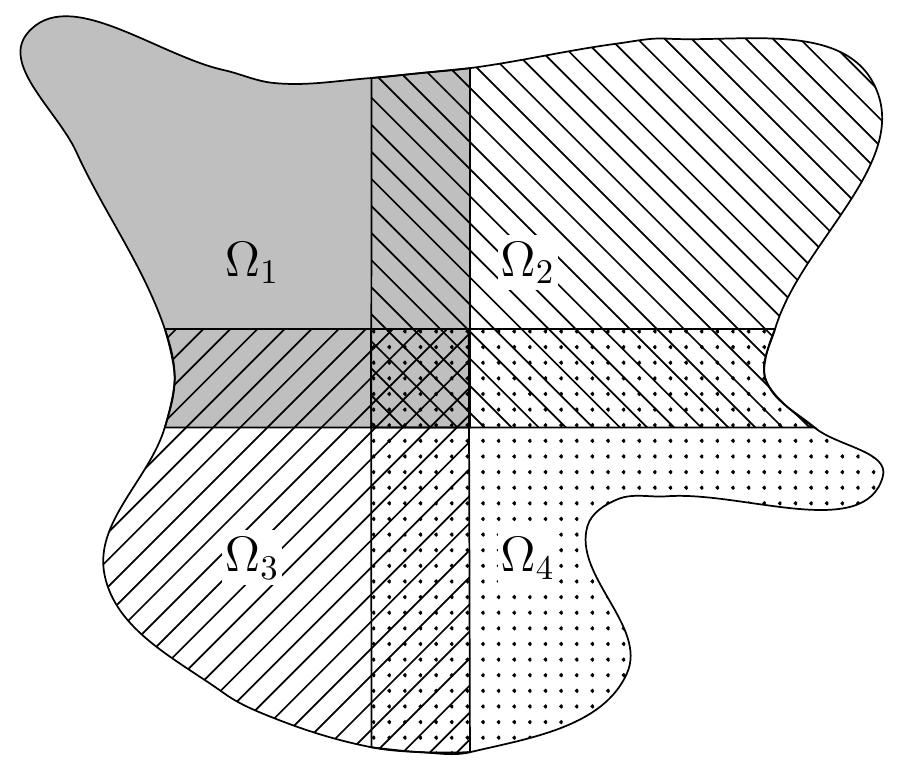}
 \includegraphics[scale=0.6]{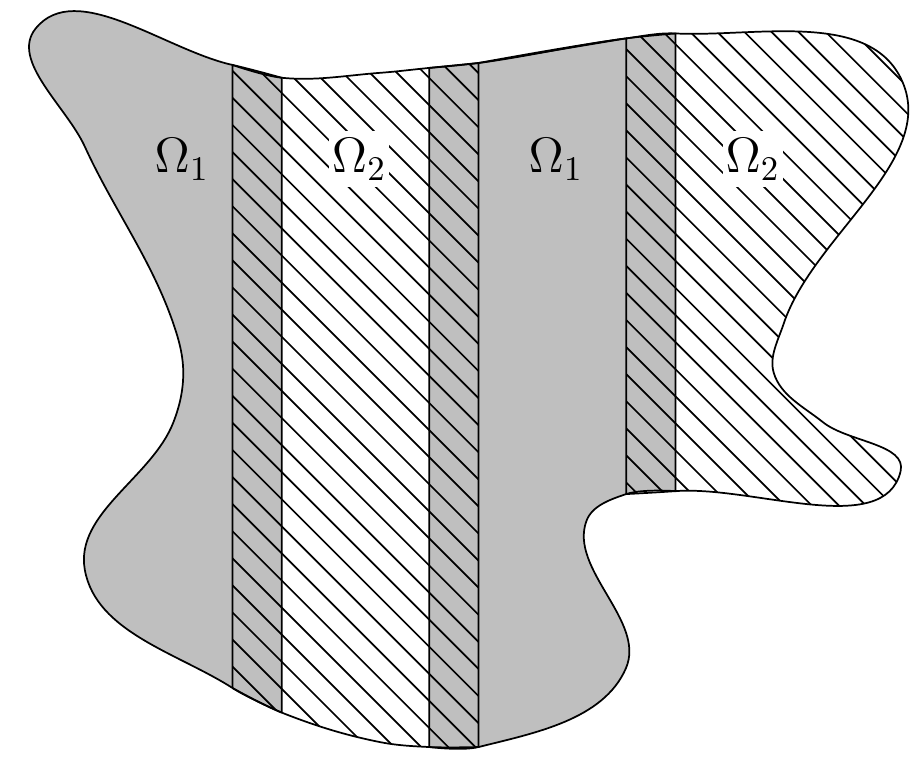}
\caption{Examples of overlapping domain decompositions $\{ \Omega_{\ell} \}_{\ell =1}^{s}$ of a domain $\Omega\subset \R^{2}$,
with $s=4$ subdomains (left) and $s=2$ subdomains that are further decomposed into families of pairwise disjoint sets (right), respectively.}
\label{fig:dom}
\end{figure}

In order to remedy this, we propose to directly introduce the domain decomposition in the time integrator via an operator splitting procedure. More precisely, let~$\{ \Omega_{\ell} \}_{\ell =1}^{s}$ be an overlapping decomposition of the spatial domain $\Omega$, as exemplified in Figure~\ref{fig:dom}. On these subdomains we introduce the partition of unity $\{ \chi_{\ell} \}_{\ell =1}^{s}$ and the operator decomposition, or splitting, 
\begin{equation} \label{eq:opdecomp}
fu= \nabla\cdot \bigl(D(u,\nabla u)\nabla u\bigr)= \sum_{\ell=1}^{s}\nabla\cdot \bigl(\chi_{\ell} D(u,\nabla u)\nabla u\bigr) = \sum_{\ell=1}^{s} f_{\ell}u.
\end{equation}
Two possible (formally) first-order integrators are then the sum splitting 
\begin{equation}\label{schemes:sumpre}
\left\{
\begin{aligned}
&v_{\ell}  =u_{n}+ sh f_{\ell} v_{\ell}, \quad \ell=1,\ldots, s,\\
& u_{n+1} = \sfrac1s\sum_{\ell=1}^{s} v_{\ell},
\end{aligned}
\right.
\end{equation}
which represents a ``quick and dirty'' scheme that is straightforward to parallelize, and the Lie splitting 
\begin{equation}\label{schemes:liepre}
\left\{
\begin{aligned}
& v_{0}     =u_{n},\\
& v_{\ell}   =v_{\ell-1}+ h f_{\ell} v_{\ell}, \quad \ell=1,\ldots, s,\\
& u_{n+1} = v_{s},
\end{aligned}
\right.
\end{equation}
which is usually more accurate but requires a further partitioning of the subdomains $\Omega_\ell$ in order to enable parallelization, as illustrated in Figure~\ref{fig:dom}. In contrast to the earlier domain decomposition based schemes, where an iterative procedure is required with possibly many instances of boundary communications, one time step of either splitting scheme only needs the solution of $s$ elliptic equations together with the communication of the data related to the overlaps. Similar splitting schemes have, e.g., been considered in the papers \cite{Arraras.2015,Hansen.2016,Mathew.1998,Vabishchevich.2013} when applied to linear, and to some extent semilinear, parabolic problems. However, there does not seem to be any analysis applicable to degenerate, or even quasilinear, parabolic equations in the literature.

Hence, the goal of this paper is twofold. First, we aim to derive a new energetic, or variational, framework that allows a proper interpretation of the operator decomposition~\eqref{eq:opdecomp} for two commonly occurring families of degenerate parabolic equations. These are the $p$-Laplace type evolutions, where the prototypical example is given by $D(u,\nabla u)=|\nabla u|^{p-2}$, and the porous medium type equations, where $D(u,\nabla u)=(p-1)|u|^{p-2}$ in the simplest case. For the porous medium application we will use the strategic reformulation 
\begin{equation*} 
fu=\Delta \alpha (u)=\sum_{\ell=1}^{s}\Delta \bigl(\chi_{\ell}\alpha (u)\bigr)=\sum_{\ell=1}^{s} f_{\ell}u
\end{equation*}
of the decomposition~\eqref{eq:opdecomp}, in order to enable an energetic interpretation. 

Secondly, we will strive to obtain a general convergence analysis for the domain decomposition based time integrators, including the sum and Lie splitting schemes. The main idea of the convergence analysis is to introduce the nonlinear Friedrich extensions of the operators $f$ and $f_{\ell}$, via our new abstract energetic framework, and then to employ a Lax-type result from the nonlinear semigroup theory~\cite{BrezisPazy.1972}.

\section{Function spaces}\label{sec:func}

Throughout the analysis $\Omega \subset \R^d$, $d\geq 1$, will be an open, connected and bounded set and the parameter $p\in (1,\infty)$ is fixed. Next, let $\{ \Omega_{\ell} \}_{\ell =1}^{s}$ be a family of overlapping subsets of $\Omega$ such that $\bigcup_{\ell =1}^s \Omega_{\ell} = \Omega$ holds. Here, each $\Omega_{\ell}$ is either an open connected set, or a union of pairwise disjoint open, connected sets $\Omega_{\ell,k}$ such that $\bigcup_{k=1}^{r}\Omega_{\ell,k} = \Omega_{\ell}$. On $\{ \Omega_{\ell} \}_{\ell =1}^{s}$ we introduce the partition of unity $\{ \chi_{\ell} \}_{\ell =1}^{s}\subset C^{\infty}(\Omega)$ such that
\begin{align*}
\chi_{\ell} (x)>0\text{ for all }x\in\Omega_{\ell},\quad \chi_{\ell} (x) = 0\text{ for all }x\in\Omega\setminus\Omega_{\ell}\quad \text{and} \quad \sum_{\ell =1}^{s} \chi_{\ell}= 1.
\end{align*} 
For details on the construction of explicit domain decompositions $\{ \Omega_{\ell} \}_{\ell =1}^{s}$ and partitions of unity $\{ \chi_{\ell} \}_{\ell =1}^{s}$ we refer to \cite[Section~3.2]{Arraras.2015} and \cite[Section~4.1]{Mathew.1998}.

The related weighted Lebesgue space $L^p(\Omega_{\ell},\chi_{\ell})$ can now be defined as the set of all measurable functions $u$ on $\Omega_{\ell}$ such that the norm
\begin{align*}
\|u\|^p_{L^p(\Omega_{\ell},\chi_{\ell})} = \int_{\Omega_{\ell}}\chi_{\ell} |u|^p \dx
\end{align*}
is finite. The space $L^p(\Omega_{\ell},\chi_{\ell})$ is a reflexive Banach space, which follows by observing that the map $G : L^p(\Omega_{\ell},\chi_{\ell}) \to L^p(\Omega_{\ell}):u\mapsto \chi_{\ell}^{\nicefrac 1p} u$ is an isometric isomorphism \cite[Chapter~1]{DrabekEtAl.1997}. We will also make frequent use of the product space $L^p(\Omega_{\ell},\chi_{\ell})^k$, equipped with the norm
\begin{align*}
\|(u_1,\ldots,u_{k})\|_{L^p(\Omega_{\ell},\chi_{\ell})^{k}}^{\Lq}= \int_{\Omega_{\ell}}\chi_{\ell} |(u_1,\ldots,u_{k})|^p \dx,
\end{align*}
which is again a reflexive Banach space~\cite[Theorem 1.23]{AdamsFournier.2003}. 

Next, let $\left(H, \ska{\cdot}{\cdot}_{H}\right)$ be a real Hilbert space and denote the space of distributions on 
$\Omega$  by $\D$. For a given $k\geq 1$ we introduce the linear operator
\begin{align*}
\delta:  H \to \D^k,
\end{align*}
which is assumed to be continuous in the following fashion.
\begin{ass}\label{ass:1}
If $\lim_{n\to\infty} u_{n}=u$ in $H$ then, 
for $j=1,\ldots,k$,
\begin{equation*}
\lim_{n\to\infty} (\delta u_n )_j(\phi) = (\delta u)_{j}(\phi)\quad \text{in } \R\quad \text{for all } \phi \in C_0^{\infty}(\Omega).
\end{equation*}
\end{ass}
As the regularity of the weights $\chi_{\ell}$ implies that $\chi_{\ell}\phi \in C^{\infty}_0(\Omega_\ell)$ for all $\phi \in C^{\infty}_0(\Omega)$, we can define the product $\chi_{\ell}\delta u$ by 
\begin{align*}
(\chi_{\ell}\delta u)_{j}(\phi)=(\delta u)_{j}(\chi_{\ell}\phi)\quad \text{for all } \phi \in C^{\infty}_0(\Omega).
\end{align*}
With this in place we can introduce our energetic spaces $V$ and $V_{\ell}$ as subspaces of $H$ given by
\begin{align*}
V &= \Bigl\{ u \in H:\text{ there exists a } v_j\in L^p(\Omega) \text{ such that }\\
& \qquad \qquad  (\delta u )_j(\phi) = \int_{\Omega} v_j \phi \dx\quad \text{for all } \phi \in C^{\infty}_0(\Omega),\ j=1,\dots,k  \Bigr\}
\end{align*}
and
\begin{align*}
V_{\ell} &= \Bigl\{ u \in H:\text{ there exists a } v_j\in L^p(\Omega_{\ell},\chi_{\ell})\text{ such that }\\
& \qquad \qquad (\chi_{\ell}\delta u )_j(\phi) = \int_{\Omega_{\ell}} v_j \chi_{\ell} \phi \dx\quad\text{for all } \phi \in C^{\infty}_0(\Omega),\ j=1,\dots,k   \Bigr\}, 
\end{align*}
respectively. On the energetic spaces we consider the operators
\begin{align*}
\delta_{p}: V \subseteq H \to L^p(\Omega)^k \quad \text{ and } \quad 
\delta_{p, \ell}: V_{\ell} \subseteq H \to L^p(\Omega_{\ell},\chi_{\ell})^k,
\end{align*}
where $\delta_{p}$ maps $u\in V$ to the corresponding $L^p(\Omega)$ functions that $\delta u$ can be represented by, and $\delta_{p,\ell}$ maps $u\in V_{\ell}$ to the corresponding $L^p(\Omega_{\ell},\chi_{\ell})$ functions that $\chi_{\ell}\delta u$ can be represented by, respectively.

\begin{lemma}\label{lem:Vintersec}
$V = \bigcap_{\ell =1}^s V_{\ell}$.
\end{lemma}

\begin{proof}
For an arbitrary $u \in V$ it follows, for $\ell = 1,\dots,s$,  that 
\begin{align*}
(\chi_{\ell}\delta u)_j (\phi) = (\delta u)_j (\chi_{\ell}\phi) = \int_{\Omega} (\delta_p u)_j \chi_{\ell}\phi \dx
\end{align*}
for every $\phi \in C_0^{\infty}(\Omega)$ and $j = 1,\dots,k$. As $(\delta_p u)_j|_{\Omega_{\ell}} \in L^p(\Omega_{\ell}) \subseteq L^p(\Omega_{\ell},\chi_{\ell} )$, we have a representation of $(\delta u)_j$ in $L^p(\Omega_{\ell},\chi_{\ell} )$, i.e., $u \in V_{\ell}$ for every $\ell = 1,\dots, s$. Hence, $V \subseteq \bigcap_{\ell =1}^s V_{\ell}$. 
		
Next, assume that $u\in\bigcap_{\ell =1}^s V_{\ell}$. Then we can write 
\begin{align*}
(\delta u)_j (\phi) = (\delta u)_j \bigl(\sum_{\ell =1}^{s} \chi_{\ell} \phi \bigr) 
	                   = \sum_{\ell =1}^{s} (\delta u)_j \left(  \chi_{\ell} \phi \right) 
	                   = \sum_{\ell =1}^{s} \int_{\Omega_{\ell} }(\delta_{p,\ell} u)_j   \chi_{\ell} \phi \dx
\end{align*}
for every $\phi \in C_0^{\infty}(\Omega)$ and $j =1,\dots,k$. Let $w_{\ell,j}$ be the zero extension of $(\delta_{p,\ell} u)_j$ to the whole of $\Omega$. We can then define the measurable function 
$v_{j}$ on $\Omega$ as $v_j = \sum_{\ell =1}^{s} \chi_{\ell} w_{\ell,j}$, which satisfies
\begin{align*}
(\delta u)_j (\phi)= \int_{\Omega} v_j  \phi \dx\quad\text{for all }\phi \in C_0^{\infty}(\Omega).
\end{align*}
Furthermore, the $L^{p}(\Omega)$ norm of $v_{j}$ can be bounded by
\begin{align*}
\|v_j\|_{L^p(\Omega)} 
	\leq \sum_{\ell=1}^{s} \bigl(\int_{\Omega_{\ell} } \chi_{\ell}^p \left|(\delta_{p,\ell} u)_j\right|^p \dx\bigr)^{\nicefrac{1}{p}}
	\leq \sum_{\ell=1}^{s}  \|\chi_{\ell}\|_{L^\infty(\Omega_{\ell})}^{\nicefrac{(p-1)}{p}} \left\|(\delta_{p,\ell} u)_j\right\|_{L^p(\Omega_{\ell}, \chi_{\ell})}.
\end{align*}
This yields that $(\delta_p u)_j = v_j \in L^p(\Omega)$ for $j =1,\dots,k$, i.e., $u\in V$ and we thereby have the identification $V=\bigcap_{\ell =1}^s V_{\ell}$. \qed
\end{proof}

\begin{lemma}
If Assumption~\ref{ass:1} holds, then the operators $\delta_p$ and $\delta_{p,\ell}$, $\ell=1,\dots,s$, are linear and closed.
\end{lemma}

\begin{proof}
The linearity of the operators is clear, since $\delta$ is a linear operator.  
Let the sequence $\seq{u}\subset V_{\ell}$ satisfy
\begin{align*}
\lim_{n\to\infty} u_{n}=u\quad\text{in }H \quad \text{and} \quad
\lim_{n\to\infty} \delta_{p,\ell} u_n=v\quad\text{in }L^p(\Omega_{\ell},\chi_{\ell})^k.
\end{align*} 
Assumption~\ref{ass:1} then yields that 
\begin{align*}
(\chi_{\ell} \delta u)_j (\phi) = \lim_{n\to \infty} (\delta u_n)_j (\chi_{\ell}\phi)	
= \lim_{n\to \infty } \int_{\Omega_{\ell}} (\delta_{p,\ell} u_n)_j \chi_{\ell} \phi	 \dx 
= \int_{\Omega_{\ell}} v_j \chi_{\ell} 	\phi \dx
\end{align*}
for every $\phi \in C_0^{\infty}(\Omega)$ and $j=1,\dots,k$. Hence, $(\chi_{\ell}\delta u)_j$ can be represented by the 
$L^p(\Omega_{\ell},\chi_{\ell})$ function $v_j$, i.e., $\delta_{p,\ell} u = v$ holds and the operator $\delta_{p, \ell}$ is therefore closed. The closedness of $\delta_{p}$ follows by the same line of reasoning. 
\qed
\end{proof}

On the energetic spaces $V$ and $V_{\ell}$, $\ell =1,\dots,s$, we define the norms
\begin{align*}
\|\cdot\|_{V}= \|\cdot\|_H + \|\delta_{p}\cdot\|_{L^p(\Omega)^k}\quad\text{and}\quad
\|\cdot\|_{V_{\ell}}= \|\cdot\|_H + \| \delta_{p, \ell}\cdot \|_{L^p(\Omega_{\ell},\chi_{\ell})^k},
\end{align*}
respectively. 

\begin{lemma}
If Assumption~\ref{ass:1} holds, then the spaces $(V, \|\cdot\|_V)$ and $(V_{\ell}, \|\cdot\|_{V_{\ell}})$, $\ell =1,\dots,s$, are reflexive Banach spaces.
\end{lemma}

\begin{proof}
Consider the reflexive Banach space $X=H\times L^p(\Omega_{\ell},\chi_{\ell})^k$,
equipped with the norm $\|(u_{1},u_{2})\|_{X}=\|u_{1}\|_{H}+\|u_{2}\|_{L^p(\Omega_{\ell},\chi_{\ell})^k}$,
and introduce the linear and isometric operator
\begin{align*}
G:V_{\ell}\to X:u\mapsto (u,\delta_{p,\ell}u).
\end{align*}
The graph of the closed operator $\delta_{p,\ell}$ coincides with the image $G(V_{\ell})$, which makes $G(V_{\ell})$ a closed linear subset of $X$. Here, $(G(V_{\ell}),\|\cdot\|_{X})$ is a reflexive Banach space \cite[Theorem 1.22]{AdamsFournier.2003} and, as $G$ is isometric, it is isometrically isomorphic to $(V_{\ell},\|\cdot\|_{V_{\ell}})$. Hence, the latter is also a reflexive Banach space. The same line of argumentation yields that $V$ is a reflexive Banach space.
\qed
\end{proof}

Hereafter, we will assume the following.
\begin{ass}\label{ass:2}
The set $V$ is dense in $H$.
\end{ass}
Under this assumption it also holds that $V_{\ell}$ is a dense subsets of $H$. By the construction of the energetic norms, one then obtains that the reflexive Banach spaces $(V, \|\cdot\|_V)$ and $(V_{\ell}, \|\cdot\|_{V_{\ell}})$ are densely and continuously embedded in $H$ and we have the following Gelfand triplets
\begin{align*}
V \overset{d}{\incl} H \cong H^* \overset{d}{\incl}  V^*
\quad \text{and}\quad 
V_{\ell} \overset{d}{\incl} H \cong H^* \overset{d}{\incl} V^*_{\ell}.
\end{align*}
Here, the density of  $H^*$ in $V^*$ and $V^*_{\ell}$, respectively, follows, e.g., by~\cite[Bemerkung I.5.14]{GGZ.1974}. For future reference, we denote the dual pairing between a Banach space $X$ and its dual $X^*$ by $\dualX{\cdot}{\cdot}$, 
and the Riesz isomorphism from $H$ to $H^{*}$ by
\begin{align*}
\gamma: H \to H^*: u\mapsto \ska{u}{\cdot}_{H}.
\end{align*}
Here, the Riesz isomorphism satisfies the relations
\begin{align*}
\dualV{\gamma u}{v}=\ska{u}{v}_{H}\quad\text{and}
\quad\dualVi{\gamma u}{v_{\ell}}=\ska{u}{v_{\ell}}_{H}
\end{align*}
for all $u\in H$, $v\in V$ and $v_{\ell}\in V_{\ell}$. 

\begin{remark}\label{rem:hm1}
Throughout the derivation of the energetic framework we have assumed that the partition of unity $\{ \chi_{\ell} \}_{\ell =1}^{s}$ consists of elements in $C^{\infty}(\Omega)$. This is somewhat restrictive from a numerical point of view, but this regularity is required if nothing else is known about the operator $\delta: H \to \D^k$. Fortunately, in concrete examples; see Sections~\ref{sec:pLap} and~\ref{sec:por}, one commonly has that $\delta(H)\subseteq H^{-1}(\Omega)^k$. If we then choose a partition of unity $\{ \chi_{\ell} \}_{\ell =1}^{s}$ in $W^{1,\infty}(\Omega)$, we have the property that $\chi_{\ell}\phi\in H^{1}_{0}(\Omega)$  for every $\phi\in H^{1}_{0}(\Omega)$, and we can once more derive the above energetic setting by testing with functions $\phi$ in $H^{1}_{0}(\Omega)$, instead of in $C^{\infty}_{0}(\Omega)$.
\end{remark}

\section{Energetic extensions of the vector fields}\label{sec:enform}

With the function spaces in place, we are now able to define the general energetic extensions of our vector fields.
\begin{ass}\label{ass:3}
For a fixed $p\in(1,\infty)$, let $\alpha: \Omega \times \R^k \to \R^k$ fulfill the properties below.
\begin{itemize}
\item[$\alpha_{1})$] The map $\alpha: \Omega \times \R^k \to \R^k$ fulfills the Carath\'{e}odory condition, i.e., $z \mapsto \alpha(x,z)$ is continuous for a.e.\ $x\in \Omega$ and $x \mapsto \alpha(x,z)$ is measurable for every $z\in \R^k$.
\item[$\alpha_{2})$] The growth condition $|\alpha(x,z)| \leq c_1 |z|^{p-1} +c_2(x) $ holds for a.e.\ $x\in \Omega$ and every $z\in \R^k$, where $c_1>0$ and $c_2\in L^{\nicefrac{p}{(p-1)}}(\Omega)$ is nonnegative.
\item[$\alpha_{3})$] The map $\alpha$ is monotone, i.e., for every $z,\tilde{z} \in \R^k$ and a.e.\ $x\in \Omega$ the inequality $(\alpha(x,z) - \alpha(x,\tilde{z}))\cdot(z - \tilde{z}) \geq 0 $ holds. 
\item[$\alpha_{4})$] The map $\alpha$ is coercive, i.e., there exists $c_3>0$ and $c_4\in L^1(\Omega)$ such that for every $z\in \R^k$ and a.e.\ $x\in \Omega$ the condition $\alpha(x,z) \cdot z \geq c_3 |z|^p - c_4(x)$ holds.
\end{itemize}
\end{ass}
Compare with~\cite[Section 26.3]{Zeidler.1989}.

We introduce the full energetic operator $F : V \to V^*$ as
\begin{align*}
\dualV{Fu}{v} = \int_{\Omega} \alpha(\delta_p u ) \cdot \delta_p v \dx\quad\text{for }u,v\in V.
\end{align*}
The operator $F$ is well defined, as $\delta_p v \in L^p(\Omega)^k$ for $v\in V$ and by ($\alpha_{2}$) 
we obtain that $\alpha(\delta_p v) \in  L^{\nicefrac{p}{(p-1)}}(\Omega)^k \cong \left(L^{p}(\Omega)^k\right)^*$.
Furthermore, we define the decomposed energetic operators $F_{\ell} : V_{\ell} \to V^*_{\ell}$, $\ell =1,\dots, s$, by
\begin{align*}
\dualVi{F_{\ell} u}{v} = \int_{\Omega_{\ell}} \chi_{\ell} \alpha(\delta_{p, \ell} u ) \cdot \delta_{p, \ell}v \dx\quad\text{for all }u,v\in V_{\ell}.
\end{align*}
These operators are well defined, as
\begin{align*}
|\dualVi{F_{\ell} u& }{v}| 
\leq \int_{\Omega_{\ell}} \chi_{\ell} (c_{1}|\delta_{p, \ell} u |^{p-1} +c_{2}) |\delta_{p, \ell}v|\dx\\
&\leq \bigl(c_{1}\bigl(\int_{\Omega_{\ell}} \chi_{\ell}|\delta_{p, \ell}u|^p\dx\bigr)^{\nicefrac{(p-1)}{p}}
                +\bigl(\int_{\Omega_{\ell}} \chi_{\ell}c_{2}^{\nicefrac{p}{(p-1)}}\dx\bigr)^{\nicefrac{(p-1)}{p}}\bigr)
\bigl(\int_{\Omega_{\ell}} \chi_{\ell} |\delta_{p, \ell}v|^p \dx\bigr)^{\nicefrac{1}{p}}
\end{align*}
is finite for every $u,v\in V_{\ell}$, due to ($\alpha_{2}$). This family of operators is a decomposition of $F$, as it fulfills
\begin{align*}
\dualV{Fu}{v}=\sum_{\ell =1}^{s} \dualVi{F_{\ell}u}{v}\quad\text{for all }u,v\in V.
\end{align*}
We can now derive the basic properties of the energetic operators.

\begin{lemma}\label{lem:energy}
If the Assumptions \ref{ass:1}--\ref{ass:3} hold and $h>0$ , then the operators $\gamma + hF: V \to V^*$ and 
$\gamma+ hF_{\ell}: V_{\ell} \to V^*_{\ell}$, $\ell=1,\ldots,s$, are strictly monotone, hemicontinuous and coercive.
\end{lemma}

\begin{proof} 
We will only derive the properties for $\gamma + hF_{\ell}$, as the same argumentation holds for $\gamma + hF$. The strict monotonicity of the operator follows using ($\alpha_{3}$), as
\begin{align*}
\dualVi{(\gamma+ hF_{\ell})u - &(\gamma + hF_{\ell})v}{u-v}= \\
                          & \ska{u-v}{u-v}_{H} + h\int_{\Omega_{\ell}} \chi_{\ell} 
                             \bigl(\alpha(\delta_{p, \ell} u) - \alpha(\delta_{p, \ell} v)  \bigr) \cdot \delta_{p, \ell} (u-v) \dx> 0
\end{align*}
holds for all $u,v\in V_{\ell}$ with $u\neq v$.

Next, we prove that $F_{\ell}$ is hemicontinuous, i.e., $t \mapsto \dualVi{F_{\ell}(u +tv)}{w}$ is continuous on $[0,1]$ for $u,v,w\in V_{\ell}$. Consider a sequence $\seq{t}$ in $[0,1]$ with limit $t$ and introduce
\begin{align*}
g(t,x)= \chi_{\ell}(x) \alpha\bigr(x,(\delta_{p, \ell} u +t\delta_{p, \ell}v) (x)\bigr) \cdot \delta_{p, \ell} w(x).
\end{align*}
As $\lim_{n\to \infty } g(t_{n},x)=g(t,x)$ holds for almost every $x\in \Omega_{\ell}$, due to ($\alpha_{1}$), and
\begin{align*}
|g(t,x)|\leq \chi_{\ell}(x) \bigl( c_{1} \bigl(|\delta_{p, \ell} u(x)|+|\delta_{p, \ell} v(x)|\bigr)^{p-1}+c_{2}(x)\bigr)
|\delta_{p, \ell} w(x)|,
\end{align*}
where the right-hand side is an $L^1(\Omega_{\ell})$ element, we obtain that
\begin{align*}
\lim_{n\to \infty} \dualVi{F_{\ell}(u +t_nv)}{w}	
    &= \lim_{n\to \infty}\int_{\Omega_{\ell}}\chi_{\ell} \alpha(\delta_{p, \ell} (u +t_nv)  ) \cdot \delta_{p, \ell} w \dx\\
    &= \dualVi{F_{\ell}(u +tv)}{w},
\end{align*}
by the dominated convergence theorem. This implies that $F_{\ell}$ is hemicontinuous, 
and the same trivially holds for $\gamma + hF_{\ell}$.
	
Last, we prove the coercivity of $\gamma + hF_{\ell}$. By assumption ($\alpha_{4}$), we have\begin{align*}
\dualVi{(\gamma + hF_{\ell}) u }{u} 
	&= \ska{u}{u}_{H} + h \int_{\Omega_{\ell}} \chi_{\ell} \alpha(\delta_{p, \ell} u  ) \cdot \delta_{p, \ell} u \dx \\
	&\geq  \|u\|_H^2 + h \int_{\Omega_{\ell}} \chi_{\ell}(c_3|\delta_{p, \ell} u |^p -c_{4})\dx\\
	&\geq\|u\|_H^2 + c_3h  \|\delta_{p, \ell} u\|_{L^p(\Omega_{\ell},\chi_{\ell})^k}^p - h\|\chi_{\ell}\|_{L^\infty(\Omega_{\ell})}\|c_4\|_{L^1(\Omega_{\ell})}
\end{align*} 
for every $u \in V_{\ell}$. Hence, we have the limit
\begin{align*}
\frac{\dualVi{(\gamma + hF_{\ell}) u }{u}}{\|u\|_{V_{\ell}}}\geq  
\min(1,c_{3}h)\frac{\|u\|_H^2 + \|\delta_{p, \ell} u\|_{L^p(\Omega_{\ell},\chi_{\ell})^k}^p}{\|u\|_H 
+ \|\delta_{p, \ell} u\|_{L^p(\Omega_{\ell},\chi_{\ell})^k} } 
- \frac{c(\chi_{\ell},c_{4})}{\|u\|_{V_{\ell}}}\to \infty,
\end{align*}
as $\|u\|_{V_{\ell}} \to \infty$, which implies the coercivity of $\gamma + hF_{\ell}$. \qed
\end{proof}

\begin{corollary}\label{cor:energy}
If the Assumptions \ref{ass:1}--\ref{ass:3} hold and $h>0$ , then the operators $\gamma + hF: V \to V^*$ and 
$\gamma+ hF_{\ell}: V_{\ell} \to V^*_{\ell}$, $\ell=1,\ldots,s$, are all bijective.
\end{corollary}

\begin{proof}
As  $\gamma + hF: V \to V^*$ and $\gamma+ hF_{\ell}: V_{\ell} \to V^*_{\ell}$ are all, by Lemma~\ref{lem:energy}, 
strictly monotone, hemicontinuous and coercive, their bijectivity follows by the Browder--Minty theorem; 
see, e.g., \cite[Theorem 26.A]{Zeidler.1989}.\qed
\end{proof}

\section{Friedrich extensions of the vector fields}\label{sec:frform}

The energetic setting is too general for the convergence analysis that we have in mind. We therefore
introduce the nonlinear Friedrich extensions of our vector fields, i.e., we restrict the domains of the energetic operators such that they become (unbounded) operators on the pivot space $H$. More precisely, we define the Friedrich extension $f: \dom(f) \subseteq H \to H$ of the full vector field by
\begin{align*} 
\dom (f) = \{u \in V : F u \in H^* \}\quad\text{and}
\quad f u =  -\gamma^{-1} Fu\quad\text{for } u\in\dom (f).
\end{align*}
Analogously, we introduce the Friedrich extensions $f_{\ell}: \dom(f_{\ell}) \subseteq H \to H$, $\ell =1,\dots,s$, of the decomposed vector fields by
\begin{align*} 
\dom (f_{\ell}) = \{u \in V_{\ell} : F_{\ell} u \in H^* \}\quad\text{and}
\quad f_{\ell} u = -\gamma^{-1} F_{\ell}u\quad\text{for  } u\in\dom (f_{\ell}).
\end{align*}

\begin{lemma} \label{lem:friedrich}
If the Assumptions \ref{ass:1}--\ref{ass:3} hold, then the operators $f:\dom(f) \subseteq H \to H$ and $f_{\ell}:\dom(f_{\ell}) \subseteq H \to H$, 
$\ell =1,\dots,s$, are all maximal dissipative.
\end{lemma}

\begin{proof}
By ($\alpha_{3}$) of Assumption~\ref{ass:3}, we have that
\begin{align*}
\ska{f_{\ell}u - f_{\ell}v}{u - v}_{H} & = -\dualVi{F_{\ell}u-F_{\ell}v}{u-v}\\
&=-\int_{\Omega_{\ell}} \chi_{\ell}\bigl(\alpha(\delta_{p,\ell} u)- \alpha(\delta_{p,\ell} v)\bigr)
\cdot \delta_{p,\ell} (u - v)\dx \leq 0
\end{align*}
for all $u,v\in\dom(f_{\ell})$, i.e., $f_{\ell}$ is dissipative. Next, for given $h>0$ and $v\in H$ one has, in virtue of Corollary~\ref{cor:energy}, that 
there exists a unique $u\in V_{\ell}$ such that $(\gamma+hF_{\ell})u=\gamma v$, or equivalently
\begin{align*}
F_{\ell} u = -\sfrac1h\, \gamma (u-v)\in H^{*}.
\end{align*}
Hence, $u\in\dom(f_{\ell})$ and $(I-hf)u=v$ in $H$, i.e., $\ran(I-hf_{\ell})=H$
and $f_{\ell}$ is therefore maximal. The same argumentation also yields that $f$ is maximal dissipative.
\qed
\end{proof}

Before we continue with our analysis we recapitulate a few properties of a general maximal dissipative 
operator $g:\dom(g)\subseteq H\to H$. The resolvent
\begin{align*}
(I-hg)^{-1}:H\to\dom(g)\subseteq H
\end{align*}
is well defined, for every $h>0$, and nonexpansive, i.e.,
\begin{align*}
\|(I-hg)^{-1}u-(I-hg)^{-1}v\|_{H}\leq\|u-v\|_{H}\quad\text{for all }u,v\in H.
\end{align*}
The latter follows directly by the definition of dissipativity. Furthermore, the resolvent 
and the related Yosida approximation $g(I - hg)^{-1}$ satisfies the following.

\begin{lemma} \label{lem:yosida}
If $g:\dom(g)\subseteq H\to H$ is maximal dissipative, then 
\begin{align*}
\lim_{h\to 0} (I - hg)^{-1}u=u\quad\text{and}\quad\lim_{h\to 0} g(I - hg)^{-1}v= gv
\end{align*}
in $H$ for every $u\in \overline{\dom(g)}$ and $v\in\dom(g)$, respectively.
\end{lemma}

The proof of Lemma~\ref{lem:yosida} can, e.g., be found in \cite[Proposition II. 3.6]{Barbu.1976} or 
\cite[Proposition 11.3]{Deimling.1985}. Next, we will relate the full vector field $f$ with its decomposition $\sum_{\ell =1}^{s} f_{\ell}$.

\begin{lemma} \label{lem:frieddom}
If the Assumptions \ref{ass:1}--\ref{ass:3} hold, then $\bigcap_{\ell =1}^s \dom(f_{\ell}) \subseteq \dom(f)$ and
$fu = \sum_{\ell =1}^{s} f_{\ell}u$ for every $u\in\bigcap_{\ell =1}^s \dom(f_{\ell})$.
\end{lemma}

\begin{proof}
Choose a $u\in\bigcap_{\ell =1}^s \dom(f_{\ell})$, then $u\in\bigcap_{\ell =1}^s V_{\ell} = V$ and the sum $z=\sum_{\ell=1}^{s} f_{\ell}u\in H$ satisfies the relation
\begin{align*}
\ska{-z}{v}_{H}=\sum_{\ell =1}^{s} \dualVi{F_{\ell}u}{v}=\dualV{Fu}{v}
\end{align*}
for all $v\in V$. Hence, $Fu\in H ^{*}$, which yields that $u\in\dom(f)$ and $fu=-\gamma^{-1} Fu =z$.
\qed
\end{proof}

Unfortunately, the set $\dom(f)$ is in general not equal to $\bigcap_{\ell =1}^s \dom(f_{\ell})$, as $u\in \dom(f)$ does not necessarily imply that $F_{\ell} u \in H^*$ for every $\ell =1,\dots,s$. This issue is well known and we will encounter it when decomposing the $p$-Laplacian; compare with Section~\ref{sec:pLap}. We will therefore assume that the mild regularity property below holds.
\begin{ass}\label{ass:f1}
$V \subseteq  \ran\bigl( I - h f |_{\bigcap_{\ell =1}^s \dom(f_{\ell})} \bigr)\quad$for all $h>0$.
\end{ass}
Under this assumption one has the following identification, which is sufficient for our convergence analysis.

\begin{lemma}\label{lem:close}
If the Assumptions \ref{ass:1}--\ref{ass:f1} hold, then the closure of $f|_{\bigcap_{\ell =1}^s\dom(f_{\ell})}$ is $f$, i.e., 
\begin{align*}
\overline{\graph\bigl( f|_{\bigcap_{\ell =1}^s \dom(f_{\ell})} \bigr)} = \graph(f).
\end{align*}
\end{lemma}

\begin{proof}
By  Lemma~\ref{lem:frieddom} and the fact that the maximal dissipative operator $f$ is closed~\cite[Proposition II.3.4]{Barbu.1976}, 
we obtain that
\begin{align*}
	\overline{ \graph\bigl(f|_{\bigcap_{\ell =1}^s \dom(f_{\ell})} \bigr)} 
	\subseteq \overline{ \graph(f)}=\graph(f).
\end{align*}
Next, choose an arbitrary $(u,fu) \in \graph(f)$. Since
\begin{align*}
u \in \dom(f) \subseteq V \subseteq  \ran\bigl( I - h f |_{\bigcap_{\ell =1}^s \dom(f_{\ell})} \bigr),
\end{align*}
we can define $v_h \in  \bigcap_{\ell =1}^s\dom(f_{\ell})$ via
\begin{align*}
v_h = (I - hf)^{-1} u = \bigl(I - h f |_{\bigcap_{\ell =1}^s \dom(f_{\ell})} \bigr)^{-1} u
\end{align*}
for every $h>0$.
By Lemma~\ref{lem:yosida}, we have the limits
\begin{align*}
\lim_{h\to 0} v_h = u\quad \text{ and }\quad \lim_{h\to 0} f v_h = \lim_{h\to 0} f(I-hf)^{-1} u  =fu\quad \text{ in } H.
\end{align*}
Hence, the set $\graph\bigl( f|_{\bigcap_{\ell =1}^s \dom(f_{\ell})} \bigr)$ is dense in 
$\graph(f)$, i.e., its closure in $H\times H$ is equal to $\graph(f)$. \qed
\end{proof}

\section{Abstract evolution equations and their approximations}\label{sec:approx}

With the Friedrich formulation of our full vector field $f:\dom(f) \subseteq H \to H$, the parabolic equations 
all take the form of an abstract evolution equations, i.e., 
\begin{equation}\label{eq:evolv}
\dot{u}=fu,\quad u(0)=\eta,
\end{equation}
on $H$. Furthermore, with the decomposition $f=\sum_{\ell =1}^s f_{\ell}$, the splitting schemes 
\eqref{schemes:sumpre} and \eqref{schemes:liepre} are given by the operators
\begin{align*}
S_{h}=\sfrac1s\,\sum_{\ell=1}^{s}\bigl(I-hsf_{\ell}\bigr)^{-1} : H \to H\quad\text{and}\quad
P_{h}= \prod_{\ell=1}^{s}\bigl(I-hf_{\ell}\bigr)^{-1}: H\to H,
\end{align*}
respectively. Here, $S^{n}_{h}\eta$ and $P^{n}_{h}\eta$ are both approximations of the exact solution $u$ at time $t=nh$.

As the resolvent of a maximal dissipative operator is well defined and nonexpansive on $H$, it is a natural starting point for a solution concept. To this end, consider the operator family $\{\ee{tf}\}_{t\geq 0}$ defined by 
\begin{equation*}
\ee{tf}\eta=\lim_{n\to\infty}\bigl(I-\sfrac tn\,f\bigr)^{-n}\eta,
\end{equation*}
where the limit is well defined in $H$ for every $\eta\in\overline{\dom(f)}$ and $t\geq 0$; see \cite[Theorem~I]{CrandallLiggett.1971}. 
The operator family $\{\ee{tf}\}_{t\geq 0}$ is in fact a (nonlinear) semigroup and each $\ee{tf}:\overline{\dom(f)}\to\overline{\dom(f)}$ is a nonexpansive operator on $H$. The unique mild solution of the evolution equation \eqref{eq:evolv} is then given by the function $u:t\mapsto \ee{tf}\eta$, which is continuous on bounded time intervals.  An extensive exposition of the nonlinear semigroup theory can, e.g., be found in \cite{Barbu.1976}.

There is a discrepancy between the domain of the solution operator, i.e., $\dom(\ee{tf})=\overline{\dom(f)}$, and 
the fact that the operators $S_{h}$ and $P_{h}$ are not necessarily invariant over it. In order to
avoid several technicalities induced by this, we will assume the following.
\begin{ass}\label{ass:f2}
The domain $\dom(f)$  is dense in $H$.
\end{ass}
As $f$ is the closure of $f|_{\bigcap_{\ell =1}^s \dom(f_{\ell})}$, one has the inclusions
\begin{equation*}
\dom(f|_{\bigcap_{\ell =1}^s \dom(f_{\ell})})\subseteq D(f)\subseteq\overline{\dom(f|_{\bigcap_{\ell =1}^s \dom(f_{\ell})})},
\end{equation*}
which implies that $\overline{\dom(f)}=\overline{\dom(f|_{\bigcap_{\ell =1}^s \dom(f_{\ell})})}$.
Hence, $\dom(f|_{\bigcap_{\ell =1}^s \dom(f_{\ell})})$ is also dense in $H$  when Assumption~\ref{ass:f2} holds.

We can now formulate the following simplified version of the Lax-type convergence 
result given in~\cite[Corollary~4.3]{BrezisPazy.1972}.

\begin{lemma}\label{lem:BrePaz}
Consider an operator family $\{G_{h}\}_{h>0}$, where each operator $G_{h}:H\to H$ is nonexpansive on $H$
and the operator family is consistent, i.e., 
\begin{align*}
\lim_{h\to 0} \sfrac 1h\,(G_{h}-I)u= fu\quad\text{in } H \quad\text{for all }u\in \cap_{\ell =1}^s \dom(f_{\ell}). 
\end{align*}
If the Assumptions~\ref{ass:1}--\ref{ass:f2} hold, then 
\begin{align*}
\lim_{n\to\infty}\sup_{t\in(0,\,T)}\bigl\| G^{n}_{t/n}\eta-\ee{tf}\eta\bigr\|_{H}=0
\end{align*}
for every $\eta\in H$ and $T<\infty$.
\end{lemma}

\begin{theorem}\label{thm:sum}
If the Assumptions~\ref{ass:1}--\ref{ass:f2} hold, then the sum splitting~\eqref{schemes:sumpre} 
is convergent in $H$, uniformly on bounded time intervals, to the mild solution of the abstract evolution 
equation~\eqref{eq:evolv}, i.e., 
\begin{align*}
\lim_{n\to\infty}\sup_{t\in(0,\,T)}\bigl\| S^{n}_{t/n}\eta-\ee{tf}\eta\bigr\|_{H}=0
\end{align*}
for every $\eta\in H$ and $T<\infty$.
\end{theorem}

\begin{proof}
As each resolvent $(I-hsf_{\ell})^{-1}$ is nonexpansive on $H$ for all values of $hs>0$, one has the bound
\begin{align*}
\|S_{h}u-S_{h}v\|_{H}\leq \sfrac 1s\, \sum_{\ell=1}^{s} \|(I-hsf_{\ell})^{-1}u-(I-hsf_{\ell})^{-1}v\|_{H}\leq \|u-v\|_{H},
\end{align*}
and $S_{h}$ is therefore nonexpansive on $H$. To validate the consistency of $\{S_{h}\}_{h>0}$, we first observe that
\begin{align*}
\sfrac 1h\,\bigl((I-hsf_{\ell})^{-1}-I\bigr)=sf_{\ell}(I-hsf_{\ell})^{-1}.
\end{align*}
The consistency can then be formulated in terms of the Yosida approximation, i.e.,
for every $u\in \cap_{\ell =1}^s \dom(f_{\ell})$ one has the limit
\begin{align*}
\sfrac1h\,(S_{h}-I)u = \sum_{\ell=1}^{s} \sfrac{1}{hs}\,\bigl((I-hsf_{\ell})^{-1}-I\bigr)u 
= \sum_{\ell=1}^{s} f_{\ell}(I-hsf_{\ell})^{-1}u\to \sum_{\ell=1}^{s} f_{\ell}u=fu
\end{align*}
in $H$, as $h\to 0$; compare with Lemma~\ref{lem:yosida}. The desired convergence is then proven 
as the hypotheses of Lemma~\ref{lem:BrePaz} hold. \qed
\end{proof}

\begin{theorem}\label{thm:lie}
If the Assumptions~\ref{ass:1}--\ref{ass:f2} hold, then the Lie splitting~\eqref{schemes:liepre} 
is convergent in $H$, uniformly on bounded time intervals, to the mild solution of the abstract evolution 
equation~\eqref{eq:evolv}, i.e., 
\begin{align*}
\lim_{n\to\infty}\sup_{t\in(0,\,T)}\bigl\| P^{n}_{t/n}\eta-\ee{tf}\eta\bigr\|_{H}=0
\end{align*}
for every $\eta\in H$ and $T<\infty$.
\end{theorem}

\begin{proof}
We once more prove convergence by validating the hypotheses of Lemma~\ref{lem:BrePaz}. The nonexpansivity
of the operator $P_{h}$ on $H$ follows trivially as every resolvent $(I-hf_{\ell})^{-1}$ has the same property. In order to validate the consistency of $\{P_{h}\}_{h>0}$, let $u\in \cap_{\ell =1}^s \dom(f_{\ell})$ and consider the telescopic expansion
\begin{equation}\label{eq:Pconist}
\sfrac1h\,(P_{h}-I)u = \sum_{\ell=1}^{s} \sfrac1h\,\bigl((I-hf_{\ell})^{-1}-I\bigr)u_{\ell,h}= 
\sum_{\ell=1}^{s} f_{\ell}(I-hf_{\ell})^{-1}u_{\ell,h},
\end{equation}
where $u_{1,h}=u$ and 
\begin{equation*}
u_{\ell,h}=(I-hf_{\ell-1})^{-1}\ldots(I-hf_{1})^{-1}u\quad \text{for }\ell=2,\ldots,s.
\end{equation*}
As the arguments of the Yosida approximations in~\eqref{eq:Pconist} are $h$ dependent, we can not directly use Lemma~\ref{lem:yosida}. Instead, we assume for the time being that the limit 
\begin{equation}\label{eq:uhlimit}
\lim_{h\to 0} \sfrac1h(u-u_{\ell,h}) =z_{\ell},\quad \text{in }H,
\end{equation}
exists. By introducing the maximal dissipative operator 
\begin{align*}
e_{\ell}: \dom(f_{\ell}) \subseteq H \to H : u\mapsto f_{\ell}u-z_{\ell},
\end{align*}
which satisfies $(I-hf_{\ell})^{-1}u_{\ell,h}=(I-he_{\ell})^{-1}(u_{\ell,h}+hz_{\ell})$, we have the reformulation
\begin{align*}
f_{\ell}(I-hf_{\ell})^{-1}u_{\ell,h} &= \sfrac1h\,(I-he_{\ell})^{-1}(u_{\ell,h}+hz_{\ell})-\sfrac1h\,(I-he_{\ell})^{-1}u\\
&\qquad+\sfrac1h\,\bigl((I-he_{\ell})^{-1}-I\bigr)u+\sfrac1h\,(u-u_{\ell,h}).
\end{align*}
By Lemma~\ref{lem:yosida} and the nonexpansivity of $(I-he_{\ell})^{-1}$, one then obtains the limit 
\begin{align*}
\|f_{\ell}&(I-hf_{\ell})^{-1}u_{\ell,h} -f_{\ell}u\|_{H}\\
&\leq \|\sfrac1h\,(I-he_{\ell})^{-1}(u_{\ell,h}+hz_{\ell})-\sfrac1h\,(I-he_{\ell})^{-1}u\|_{H}\\
&\qquad+\|\sfrac1h\,\bigl((I-he_{\ell})^{-1}-I\bigr)u-e_{\ell} u\|_{H} + \|\sfrac1h\,(u-u_{\ell,h}) + e_{\ell} u - f_{\ell} u\|_{H}\\
&\leq \|-\sfrac1h\,(u-u_{\ell,h})+z_{\ell}\|_{H}+\|e_{\ell}(I-he_{\ell})^{-1}u-e_{\ell}u\|_{H}\\
&\qquad+ \|\sfrac1h\,(u-u_{\ell,h})-z_{\ell}\|_{H}\to 0,\quad\text{as }h\to 0.
\end{align*}
Hence, if \eqref{eq:uhlimit} exists then $ \lim_{h\to 0}f_{\ell}(I-hf_{\ell})^{-1}u_{\ell,h}=f_{\ell}u$. Furthermore,
if \eqref{eq:uhlimit} exists for every $\ell=1,\ldots,s$, then $\lim_{h\to 0}1/h\,(P_{h}-I)u=fu$ in $H$.

The limit \eqref{eq:uhlimit} obviously exists for $\ell=1$. If it exists for $\ell=k$ then it also exists for $\ell=k+1$, as
\begin{align*}
\sfrac1h\,(u-u_{k+1,h})&=\sfrac1h\,(u-u_{k,h})-\sfrac1h\,\bigl((I-hf_{k})^{-1}-I\bigr)u_{k,h}\\
&=\sfrac1h\,(u-u_{k,h})-f_{k}(I-hf_{k})^{-1}u_{k,h}\to z_{k}-f_{k}u
\end{align*}
in $H$, as $h\to 0$. By induction, the limit~\eqref{eq:uhlimit} exists for every $\ell=1,\ldots,s$, and $\{P_{h}\}_{h>0}$ is therefore consistent.\qed
\end{proof}

\begin{remark}
The results can be extended to perturbed equations $\dot{u}=(f+g)u$, e.g., arising if a lower-order advection or reaction term is added to the diffusion process. Here, $g$ and $f+g$ are both assumed to satisfy a shifted dissipativity condition of the form
\begin{align*}
\ska{gu - gv}{u - v}_{H} \leq M[g] \|u-v\|_{H}^{2}\quad\text{for all }u,v\in\dom(g),
\end{align*}
with $M$ being a nonnegative constant, and the range condition $\ran(I-hg)=H$ for $h\in(0,1/M)$.
This is, e.g., satisfied when $g:H\to H$ is Lipschitz continuous. 
More elaborate perturbation examples are given in \cite[Section~II.3.2]{Barbu.1976}.
For these perturbed evolution equations, one has convergence for the modified splitting schemes, with a single step given by
\begin{align*}
\tilde{S}_{h}= (I-hg)^{-1}S_{h}\quad\text{and}\quad \tilde{P}_{h}= (I-hg)^{-1}P_{h},
\end{align*}
respectively. If $g:H\to H$ is in addition Lipschitz continuous, then convergence is also obtained for the semi-implicit schemes
\begin{align*}
\hat{S}_{h}= (I+hg)S_{h}\quad\text{and}\quad \hat{P}_{h}= (I+hg)P_{h}.
\end{align*}
The convergence of the modified schemes follow just as for the proof of Theorem~\ref{thm:lie} together with 
the fact that \cite[Corollary~4.3]{BrezisPazy.1972} is valid for operators $G_{h}$ that have Lipschitz constants of the form $1+Ch$.
\end{remark}

\section{Parabolic equations of p-Laplace type}\label{sec:pLap}

As a first problem class we consider the parabolic equations of $p$-Laplace type with 
homogeneous Neumann boundary conditions, i.e., 
\begin{equation}\label{eq:pLap}
\begin{cases}
\partial u/ \partial t = \nabla \cdot \alpha(\nabla u)  &\text{in } \Omega\times (0,T),\\
\alpha(\nabla u)\cdot n =0 &\text{on } \partial \Omega\times (0,T),\\
u(0) =\eta  &\text{in } \Omega.
\end{cases}
\end{equation}
The domain $\Omega\subset \R^d$ is assumed to have a locally Lipschitz boundary $ \partial \Omega$, and the map $\alpha:\Omega\times\R^{d}\to \R^{d}$ satisfies Assumption~\ref{ass:3} for a given $p\geq 2$. The classical $p$-Laplacian is then given by 
\begin{equation*}
\alpha(x,z)=|z|^{p-2}z.
\end{equation*}
After multiplication with $v$ and a subsequent integration by parts, the variational form of~\eqref{eq:pLap} and its decomposition is formally given by
\begin{equation}\label{eq:pLaplvar}
(\partial u/ \partial t,v)_{L^2(\Omega)} =  - \int_{\Omega}\alpha (\nabla u)\cdot \nabla v \dx = 
- \sum_{\ell=1}^{s}  \int_{\Omega_{\ell}}\chi_{\ell} \alpha (\nabla u)\cdot \nabla v \dx.
\end{equation}
Here, we have introduce a domain decomposition $\{\Omega_{\ell}\}_{\ell =1}^s$, where $\bigcup_{\ell =1}^s \Omega_{\ell}=\Omega$, together with a partition of unity $\{\chi_{\ell}\}_{\ell=1}^{s}$ chosen in $W^{1,\infty}(\Omega)$; compare with Remark~\ref{rem:hm1}. 

In order to fit the variational form into the abstract setting of Sections~\ref{sec:enform}, we choose the pivot space $H = L^2(\Omega)$ and the operator $\delta$ as the distributional gradient
\begin{align*}
\delta: L^2(\Omega)  \to \D^d: u \mapsto \nabla u.
\end{align*}
This choice of $\delta$ fulfills the continuity Assumption~\ref{ass:1}, since for a convergent sequence $\seq{u}\subset L^2(\Omega)$ and an arbitrary $\phi \in C^{\infty}_0(\Omega)$ one can write
\begin{align*}
\lim_{n\to \infty} (D_j u_n) (\phi)  = -\lim_{n\to \infty} \int_{\Omega} u_n D_j \phi \dx = - \int_{\Omega} u D_j \phi \dx = (D_j u) (\phi) 
\end{align*} 
for every$j=1,\dots,d$, where $D_j$ is the $j$-th partial derivative in a distributional sense.
The space $V$ is then
\begin{align*}
V = \bigl\{ u\in L^2(\Omega): \nabla u \in L^p(\Omega)^d \bigr\}.
\end{align*}
A bootstrap argument using the Sobolev embedding theorem yields the identification $V = \Wp$. Since $\Wp$ is dense in $L^2(\Omega)$, Assumption~\ref{ass:2} is also fulfilled. 

With these choices, $\delta_{p}u$ is simply the weak gradient of $u\in\Wp$ and we obtain the standard energetic form $F :V \to V^*$ of $p$-Laplace type vector fields, i.e., 
\begin{align*}
\dualV{Fu}{v}  = \int_{\Omega}\alpha (\nabla u)\cdot \nabla v \dx.
\end{align*}
The domain of the corresponding Friedrich extension can be written as
\begin{align*}
\dom(f) &= \Bigl\{u \in \Wp:\text{ there exists a }z\in L^2(\Omega) \text{ such that }\\
& \qquad \qquad -\int_{\Omega}\alpha (\nabla u)\cdot \nabla v \dx=\int_{\Omega} zv\dx \quad\text{for all } v \in \Wp \Bigr\}, 
\end{align*}
and $fu$ is given by the weak divergence of $\alpha(\nabla u)$. The same characterization can be made for $F_{\ell}$ and $f_{\ell}$, respectively. Applying Lemma~\ref{lem:friedrich} the operators $f$ and $f_{\ell}$, $\ell =1,\dots,s$, are maximal dissipative and Lemma~\ref{lem:frieddom} yields that
\begin{align*}
\bigcap_{\ell =1}^{s} \dom(f_{\ell}) \subseteq \dom(f)\quad \text{ and }\quad
fu = \sum_{\ell =1}^{s} f_{\ell}u\quad\text{for }u\in\bigcap_{\ell =1}^{s} \dom(f_{\ell}).
\end{align*}
Validation of Assumption~\ref{ass:f2} requires further structure of the map $\alpha$. For the classical $p$-Laplacian the related $\alpha$ is continuously differentiable and $\alpha(0)=0$, which implies that $C_0^{\infty}(\Omega)$ is a subset of $\dom(f)$. Hence, $\dom(f)$ is dense in $L^2(\Omega)$ and Assumption~\ref{ass:f2} is valid in this context. Finally, if Assumption~\ref{ass:f1} holds then the convergence results from Section \ref{sec:approx} can directly be applied. 

\begin{figure}
 \centering 
 \includegraphics[scale=0.7]{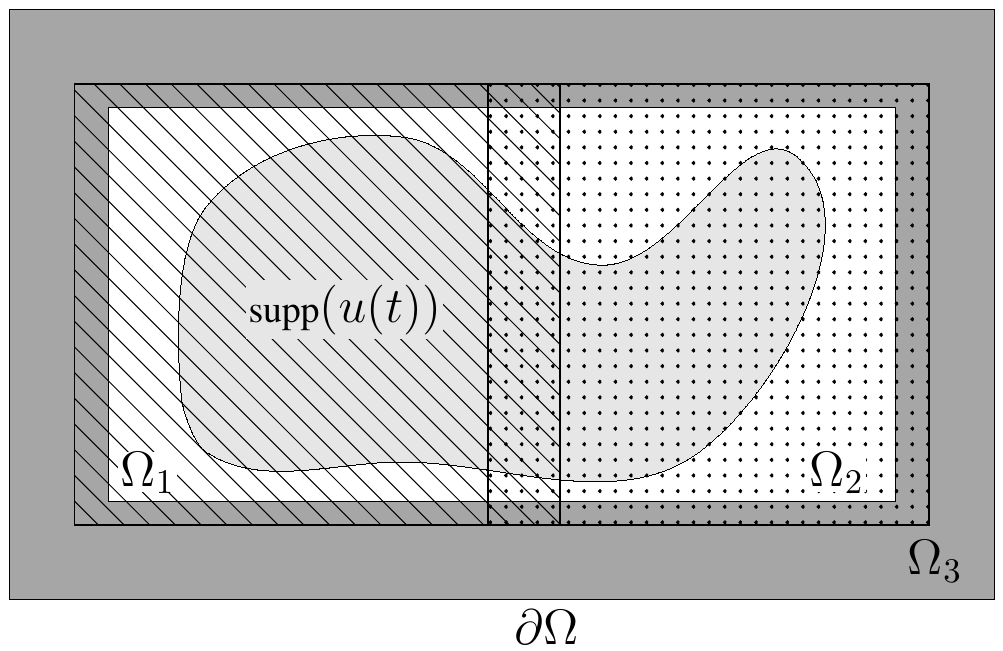}
 \caption{An example of a domain decomposition $\{\Omega_{\ell}\}_{\ell =1}^3$ that fulfills \eqref{ass:dom}.}
 \label{fig:pLaplace}
\end{figure}

Apart from the special cases when $d=1$ or $p=2$, the domains $\dom(f)$ of $p$-Laplace type vector fields can not be expected to coincide with $\bigcap_{\ell =1}^{s} \dom(f_{\ell})$. The issue is that for an element $u \in \dom(f)$ one has
\begin{align*}
f_{\ell}u=\nabla\cdot \bigl(\chi_{\ell} \alpha(\nabla u) \bigr)=  \nabla \chi_{\ell} \cdot\alpha(\nabla u) 
+ \chi_{\ell} \nabla\cdot \alpha(\nabla u),
\end{align*}
where the function $\alpha(\nabla u)$ only lies in $L^{\nicefrac{p}{(p-1)}}(\Omega)^{d}$, with $p>2$. The term $f_{\ell}u$ is therefore, in general, not an $L^{2}(\Omega)$ function. In order to give a possible setting for which Assumption~\ref{ass:f1} is valid, we assume that the domain decomposition $\{\Omega_{\ell}\}_{\ell =1}^s$ is chosen such that
\begin{equation}\label{ass:dom}
\text{closure}(\bigcup_{\ell =1}^{s-1}\Omega_{\ell})\setminus\partial\Omega = \emptyset.
\end{equation}
That is, the subdomain $\Omega_{s}$ separates the boundary $\partial\Omega$ from the other subdomains; as illustrated in Figure~\ref{fig:pLaplace}.

\begin{lemma} 
Consider a domain decomposition $\{\Omega_{\ell}\}_{\ell =1}^s$ that satisfies \eqref{ass:dom} and with subdomains $\Omega_{\ell}$, $\ell=1,\ldots,s-1$, that all have the segment property. If $p\geq 2$ in addition satisfies $p>(d+1)/2$ and the map $\alpha$ fulfills Assumption~\ref{ass:3}($\alpha_{2}$) with $c_{2}\in L^{2}(\Omega)$, then the Friedrich extension $f$ of a $p$-Laplace type vector field and its decomposition into the operators $f_{\ell}$ fulfill Assumption~\ref{ass:f1}.
\end{lemma}

\begin{proof}
For an arbitrary $g \in V=\Wp$ there exists a unique $u\in \dom(f)$ such that
$u-hfu = g$ and Assumption~\ref{ass:f1} is then valid if $u\in\bigcup_{\ell =1}^s\dom(f_{\ell})$. To prove this, we first observe that  $fu=\nabla \cdot \alpha(\nabla u) = (u-g)/h \in\Wp$ and $\Wp \incl L^r(\Omega)$ for some $r > dp/(p-1)$, as $p \geq 2$ and $p > (d+1)/2$. Hence, \cite[Theorem 2 and Remarks pp.~829--830]{diBenedetto.1983} implies that $\nabla u$  is locally H\"older continuous on $\Omega$ and we obtain that
\begin{align*}
\alpha(\nabla u)|_{\Omega_{int}}\in L^{2}(\Omega_{int})^{d}
\end{align*}
for every open domain $\Omega_{int}$ such that $\overline{\Omega}_{int}\subset\Omega$. 

As $u\in\dom(f)$, we have the integration by parts
\begin{align} \label{eq:intbypar}
-\int_{\Omega} \alpha (\nabla u)\cdot \nabla w\dx
= \int_{\Omega}  \nabla\cdot\alpha (\nabla u) w\dx
\end{align} 
for every $w\in\Wp$. Due to the  extra interior regularity of $\alpha(\nabla u)$ we can, e.g., extend~\eqref{eq:intbypar} to 
all $w=w_{1}+w_{2}$, where $w_{1}\in \Wp$ and $w_{2}\in H^{1}(\Omega)$ is a.e.\ zero on $\Omega\setminus\Omega_{int}$
for some open subdomain $\Omega_{int}$ that has the segment property and fulfills $\overline{\Omega}_{int}\subset\Omega$.
The latter implies that $w_{2}$ is the zero extension of $w_{2}|_{\Omega_{int}}\in H^{1}_0 (\Omega_{int})$; see, e.g., \cite[Theorem 5.29]{AdamsFournier.2003}.  

Next, let $v\in V_{\ell}\subset L^{2} (\Omega)$, for $\ell=1,\ldots,s$, and consider $\chi_{\ell} v \in L^2(\Omega)$. Here, 
\begin{align*}
D_{j}(\chi_{\ell} v) (\phi) &= D_{j}( v) (\chi_{\ell}\phi) + \int_{\Omega} (D_{j}\chi_{\ell}) v\phi\dx= \int_{\Omega_{\ell}} \bigl(\chi_{\ell} (\delta_{p,\ell}v)_{j}+(D_{j}\chi_{\ell})v\bigr)\phi \dx
\end{align*}
for every $\phi \in C^\infty_0 (\Omega)$, i.e.,  $\chi_{\ell} v \in H^1(\Omega)$ and $\chi_{\ell} v = 0$ a.e.\ on $\Omega\setminus \Omega_{\ell}$. If $\ell<s$ then $\chi_{\ell} v|_{\Omega_{\ell}}\in H^1_{0}(\Omega_{\ell})$. 

For $\ell=1,\ldots,s-1$, we can test with $w=\chi_{\ell}v$ and integrate by parts~\eqref{eq:intbypar}. Writing out $\nabla(\chi_{\ell}v)$ and rearranging the terms gives us 
\begin{align*} 
-\int_{\Omega_{\ell}} \chi_{\ell}\alpha (\nabla u)\cdot \delta_{p,\ell}v\dx =
 \int_{\Omega} \bigl(\chi_{\ell}\nabla\cdot \alpha (\nabla u)+\nabla\chi_{\ell}\cdot  \alpha (\nabla u)\bigr)v\dx,
\end{align*}
i.e., $u\in\bigcap_{\ell=1}^{s-1} \dom(f_{\ell})$, as the integrand on the right-hand side is in $L^{2}(\Omega)$.

It remains to prove that $u$ lies in $\dom(f_{s})$.  As the closure of $\bigcup_{\ell =1}^{s-1}\Omega_{\ell}$ does not intersect the outer boundary $\partial\Omega$, we can choose an open subset $\Omega_{out}\subset\Omega_{s}$ such that $\chi_{s}=1$ on 
$\Omega_{out}$, its boundary $\partial \Omega_{out}$ is locally Lipschitz continuous and $\partial\Omega\subset\partial\Omega_{s}$.
Let $v\in V_{s}$, then $\chi_{s}\delta_{p,s}v=\nabla v$ a.e.\ on $\Omega_{out}$ and $\chi_{s}v|_{\Omega_{out}}=v|_{\Omega_{out}}\in W^{1,p}(\Omega_{out})$. The local Lipschitz continuity of  $\partial \Omega_{out}$ implies, e.g., via \cite[Theorem 5.24]{AdamsFournier.2003}, that there exists an extension $w_{1}\in W^{1,p}(\Omega)$ such that $w_{1}=\chi_{s}v$ a.e.\ on $\Omega_{out}$.  Furthermore, $w_{2}=\chi_{s}v-w_{1}\in H^{1}(\Omega)$ is zero a.e. on  $\Omega_{out}$, i.e., it is a zero extension of an $H^{1}_0 (\Omega_{int})$ function on some subdomain $\Omega_{int}$, with $\partial\Omega_{int}\subset \Omega_{out}$. For every $v\in V_{s}$ we therefore have a partitioning of the form $w=\chi_{s}v=w_{1}+w_{2}$ and the integration by parts \eqref{eq:intbypar} is well defined for $\ell=s$. By the same argumentation as for $\ell<s$, one obtains that $u$ lies in $\dom(f_{s})$.
\qed
\end{proof}

\begin{remark}
From a numerical perspective the construction~\eqref{ass:dom} with a separating subdomain $\Omega_{s}$ is suboptimal for general time dependent PDEs, as it may increase the amount of communication in the implementation of scheme. However, as discussed in Section~\ref{sec:intro}, we are foremost interested in the approximation of solutions with compact support in $\Omega$. Hence, for sufficiently short time intervals $(0,T)$ there is obviously no communication related to $\Omega_{s}$; as exemplified in Figure~\ref{fig:pLaplace}.
\end{remark}

\section{Parabolic equations of porous medium type}\label{sec:por}

A second problem class that fits into our abstract setting is the parabolic equations of porous medium type with homogeneous Dirichlet boundary conditions, i.e., 
\begin{equation}\label{eq:pme}
\begin{cases}
\partial u/ \partial t = \Delta \alpha(u)  &\text{in } \Omega\times (0,T),\\
\alpha(u) =0 &\text{on } \partial \Omega\times (0,T),\\
u(0) =\eta  &\text{in } \Omega.
\end{cases}
\end{equation}
Here, the domain $\Omega\subset \R^d$ is assumed to have a locally Lipschitz boundary $ \partial \Omega$, and the map $\alpha:\Omega\times\R\to\R$ fulfills Assumption~\ref{ass:3} for a given $p$ that satisfies 
\begin{align*}
p\in (1,\infty)\quad\text{if }d\leq 2,\quad\text{and}\quad p\in [2d/(d+2), \infty)\quad\text{if }d>2.
\end{align*}
This restriction on $p$ is made in order to assure the embedding
\begin{equation}\label{eq:H1inLq}
H_0^1(\Omega)\overset{d}{\incl} L^{\nicefrac{p}{(p-1)}}(\Omega),
\end{equation}
which is central in our forthcoming analysis. The standard porous medium equation is then given by
\begin{align*}
\alpha(x,z)=|z|^{p-2}z,\quad\text{with }p\geq 2,
\end{align*}
and the fast diffusion equation is obtained for the same $\alpha$, but with $1<p<2$; see~\cite{Vasquez.2007}. The two-phase Stefan problem \cite[Section~5.10]{Friedman.1988} follows by choosing 
\begin{align*}
\alpha (x,z)  = 
\begin{cases}
a(z+1)  &\text{ for } z \leq -1\\
0  &\text{ for } z \in (-1,1)\\
b(z-1)  &\text{ for } z \geq 1\, ,
\end{cases}
\end{align*}
where $a,b>0$, and Assumption~\ref{ass:3} is then valid for $p=2$. 

After multiplying \eqref{eq:pme} by $w$, where $-\Delta w = v$ in $\Omega$ and $w=0$ on $\partial\Omega$, and integrating by parts twice, the variational form of~\eqref{eq:pme} and its decomposition is formally 
\begin{equation}\label{eq:pmevar}
\int_{\Omega}\frac{\partial u}{\partial t}\,(-\Delta)^{-1}v\dx = 
-\int_{\Omega}\alpha (u) v \dx = 
- \sum_{\ell=1}^{s} \int_{\Omega_{\ell}}\chi_{\ell} \alpha (u) v \dx.
\end{equation}
Above, we have once more introduced a domain decomposition $\{\Omega_{\ell}\}_{\ell =1}^s$ of $\Omega$ together with a partition of unity $\{\chi_{\ell}\}_{\ell=1}^{s}$. 

With the proper interpretation, the left-hand side of \eqref{eq:pmevar} is given by the inner product on
$H^{-1}(\Omega)$; compare with \cite[Bemerkung III.1.13]{GGZ.1974}. The formal variational formulation~\eqref{eq:pmevar} therefore leads us to choosing the pivot space $H=H^{-1}(\Omega)$ and the operator
\begin{align*}
\delta:H^{-1}(\Omega)  \to \D: u \mapsto u.
\end{align*}
The operator $\delta$ obviously fulfills the continuity Assumption~\ref{ass:1}. The space $V$ is now
\begin{align*}
V &=  \Bigl\{ u \in H^{-1}(\Omega) :\text{ there exists a } v\in L^p(\Omega) \text{ such that }\\
&  \qquad\qquad \dualHH{u}{\phi} = \int_{\Omega} v \phi \dx\quad\text{for all } \phi \in H^{1}_0(\Omega)\Bigr\}
= \bigl(L^{\nicefrac{p}{(p-1)}}(\Omega)\bigr)^*,
\end{align*}
and as before $\delta_p u = v$, where $v$ is the unique function stated in the definition of $V$. By the embedding~\eqref{eq:H1inLq} and~\cite[Bemerkung I.5.14]{GGZ.1974}, we obtain that 
\begin{align*}
\bigl(L^{\nicefrac{p}{(p-1)}}(\Omega)\bigr)^*\overset{d}{\incl} H^{-1}(\Omega),
\end{align*} 
i.e., Assumption~\ref{ass:2} is fulfilled. With these choices, we have the energetic form $F :V \to V^*$ given by
\begin{align*}
\dualV{Fu}{v}  = \int_{\Omega}\alpha (\delta_{p} u)\delta_{p} v \dx.
\end{align*}

In order to characterize the Friedrich operator $f$, we introduce the Dirichlet Laplacian $-\Delta:H^{1}_{0}(\Omega)\to H^{-1}(\Omega)$, where
\begin{align*}
\dualHH{-\Delta u}{v} = \int_{\Omega} \nabla u\cdot\nabla v\dx\quad\text{for all } u,v\in H^{1}_{0}(\Omega).
\end{align*}
As $-\Delta$ is the Riesz isomorphism from $H^{1}_{0}(\Omega)$ to $H^{-1}(\Omega)$,
the inner product on $H^{-1}(\Omega)$ satisfies
\begin{align*}
\ska{u}{v}_{H^{-1}(\Omega)} 
&= \sfrac14\bigl(\|u+v\|_{H^{-1}(\Omega)}-\|u-v\|_{H^{-1}(\Omega)}\bigr)\\
&=\sfrac14\bigl(\|(-\Delta)^{-1}(u+v)\|_{H^{1}_{0}(\Omega)}-\|(-\Delta)^{-1}(u-v)\|_{H^1_{0}(\Omega)}\bigr)\\
&=\ska{(-\Delta)^{-1}u}{(-\Delta)^{-1}v}_{H^{1}_{0}(\Omega)}\\
&=\dualHH{u}{(-\Delta)^{-1}v}
\end{align*}
for all $u,v\in H^{-1}(\Omega)$; compare with \cite{EmmrichSiska.2012}. Next, for $u\in\dom(f)$ there exists a $z\in H^{-1}(\Omega)$ such that
\begin{align*}
-\int_{\Omega}\alpha (\delta_{p} u)\,\delta_{p} v \dx=(z,v)_{H^{-1}(\Omega)}=\dualHH{v}{(-\Delta)^{-1}z}
\end{align*}
for all $v \in \bigl(L^{\nicefrac{p}{(p-1)}}(\Omega)\bigr)^*$, or equivalently 
\begin{align*}
-\int_{\Omega}\alpha (\delta_{p} u)\,w \dx= \int_{\Omega} w \,(-\Delta)^{-1} z\dx
\quad\text{ for all } w \in L^{p}(\Omega).
\end{align*}
Hence, $-\alpha (\delta_{p} u)=(-\Delta)^{-1}z\in H^{1}_{0}(\Omega)$; see, e.g., \cite[Lemma~3.31]{AdamsFournier.2003}, and we obtain the characterization 
\begin{align*}
\dom(f)=\bigl\{u\in \bigl(L^{\nicefrac{p}{(p-1)}}(\Omega)\bigr)^*:\alpha (\delta_{p} u)\in H^{1}_{0}(\Omega)\bigr\},
\end{align*}
and $fu=\Delta\alpha (\delta_{p} u)$ for $u\in\dom(f)$.

Analogously to Section~\ref{sec:pLap}, we have $\ran(\delta)= H^{-1}(\Omega) \subset \D$ and we can therefore allow a partition of unity $\{\chi_{\ell}\}$ in $W^{1,\infty}(\Omega)$. The spaces $V_{ \ell}$, $\ell =1,\dots,s$, are then
\begin{align*}
V_{\ell}&= \big\{ u \in H^{-1}(\Omega): \text{ there exists a } v \in L^p(\Omega_{\ell},\chi_{\ell}) \text{ such that }\\
& \qquad\qquad  \dualHH{\chi_{\ell} u}{\phi} = \int_{\Omega_{\ell}} \chi_{\ell} v \phi \dx\quad\text{for every } \phi \in H^{1}_0(\Omega) \big\}.
\end{align*}
Again, we write $\delta_{p, \ell} u$ for the unique $L^p(\Omega_{\ell},\chi_{\ell})$ function $v$ from this definition.

After introducing $F_{\ell}$ and $f_{\ell}$, as described in Sections~\ref{sec:enform}, we have by Lemmas~\ref{lem:friedrich} and~\ref{lem:frieddom} that the operators $f$ and $f_{\ell}$, $\ell =1,\dots,s$, are maximal dissipative and
\begin{align*}
fu = \sum_{\ell =1}^{s} f_{\ell}u\quad\text{for }u\in\bigcap_{\ell =1}^{s} \dom(f_{\ell})\subseteq \dom(f).
\end{align*}
Instead of Assumption~\ref{ass:f1} we can prove the stronger condition 
\begin{align*}
\bigcap_{\ell =1}^s \dom(f_{\ell}) = \dom(f).
\end{align*}
To prove the equality take an arbitrary $u\in \dom(f)$. Since $\alpha(\delta_p u )\in H_0^1(\Omega)$, we also have
that $\chi_{\ell}\alpha(\delta_p u ) \in H_0^1(\Omega)$ for every weight function $\chi_{\ell} \in W^{1,\infty}(\Omega)$ and
\begin{align*} 
-\int_{\Omega_{\ell}} \chi_{\ell}\alpha (\delta_{p} u)\, \delta_{p,\ell}v\dx &=\dualHH{v}{-\chi_{\ell}\alpha (\delta_{p} u)}\\
&= \ska{\Delta\bigl(\chi_{\ell}\alpha (\delta_{p} u)\bigr)}{v}_{H^{-1}(\Omega)}\quad \text{ for all } v\in V_{\ell}.
\end{align*}
That is, $u$ also lies in $\dom(f_{\ell})$ for $\ell =1,\dots,s$.

Assumption~\ref{ass:f2} requires some further regularity of the map $\alpha$ and the validation that $\alpha (\delta_{p} u)$ vanishes on the boundary $\partial\Omega$. For the porous medium equation and the two-phase Stefan problem one has that $\alpha(\phi)\in H^{1}_{0}(\Omega)$ for every $\phi\in C_0^{\infty}(\Omega)$. The set of functionals of the form $v\mapsto\int_{\Omega}u v\dx$, where $u\in C_0^{\infty}(\Omega)$ and $v\in H^{1}_{0}(\Omega)$, is therefore a subset of $\dom(f)$. It is also a dense subset of $H^{-1}(\Omega)$, as $C_0^{\infty}(\Omega)$ is dense in $L^{2}(\Omega)$ and $L^{2}(\Omega)^*$ is dense in $H^{-1}(\Omega)$. Hence, Assumption~\ref{ass:f2} is valid for these two prototypical examples, and the convergence results of Section~\ref{sec:approx} hold.

\begin{remark}
The variational setting of porous medium type equations, with $H^{-1}(\Omega)$ as pivot space, is by no means standard. However, it enables a clear-cut way of introducing the related Friedrich operator. The variational setting has, e.g., been proposed in \cite[Bemerkung I.5.14]{GGZ.1974}. It has also been employed in \cite{EmmrichSiska.2012} when proving convergence of finite element/implicit Euler approximations for the porous medium equation, on its very weak form. Note that the standard approach to prove that $\Delta\alpha$ is a maximal dissipative operator on $H^{-1}(\Omega)$ is to directly observe that it is the gradient of a convex function; see \cite[Example 3]{Brezis.1971}.
\end{remark}

\begin{acknowledgements}
Part of this study was conducted during Hansen's guest research stay at the Institut f\"{u}r Mathematik, TU Berlin. Hansen would like to thank Etienne Emmrich for enabling this inspiring stay. 
\end{acknowledgements}

\end{document}